\documentclass[12pt,oneside,reqno]{amsart}

\usepackage{amsthm, amssymb, amsmath, amsfonts}
\usepackage{enumerate}
\usepackage[colorlinks]{hyperref}

\newtheorem{thm}{Theorem}[section]

\newtheorem{defi}[thm]{Definition}
\newtheorem{lem}[thm]{Lemma}
\newtheorem{prop}[thm]{Proposition}
\newtheorem{corl}[thm]{Corollary}
\newtheorem{rem}[thm]{Remark}

\allowdisplaybreaks
\title[ Weak containment of representation]{Weak containment of representation on topological groupoids}
\author[K. N. Sridharan]{K. N. Sridharan}
\address{K. N. Sridharan,\newline\indent Department of Mathematics,\newline\indent Indian Institute of Technology Delhi,\newline\indent New Delhi - 110016, India.}
\email{sreedu242@gmail.com}
\author[N. S. Kumar]{N. Shravan Kumar}
\address{N. Shravan Kumar,\newline\indent Department of Mathematics,\newline\indent Indian Institute of Technology Delhi,\newline\indent New Delhi - 110016, India.}
\email{shravankumar.nageswaran@gmail.com}
\begin{document}
\begin{abstract}
    Let $G$ be a second-countable, locally compact Hausdorff groupoid equipped with a Haar system $\lambda$. This paper investigates the weak containment of continuous unitary representations of groupoids, formulated in terms of Hilbert bundles over the unit space $G^{0}$. We show that both induction and inner tensor product of representations preserve weak containment. Additionally, we introduce the notion of a topological invariant mean on $G/H$ and explore its connection to amenability.  We establish a groupoid analogue  of Greenleaf's theorem.  Finally, we provide independent results concerning the restriction of induced representations for continuous unitary representations of  clopen wide subgroupoids $H\subseteq G$ with discrete unit space and closed transitive wide subgroupoids of compact transitive groupoids.
\end{abstract}
\maketitle

\section{Introduction}
The notion of weak containment introduced by Fell (see \cite{fell1960dual,fell1962weak,fell1964weak}) plays an important role in the theory of representation of locally compact groups. In \cite{fell1960dual}, the author briefly discusses the properties of groups whose left regular representation weakly contains every irreducible representation. Furthermore, \cite{fell1962weak} demonstrates that induced representations preserve the weak containment property and establishes a relationship between the inner tensor product of representations and weak containment. Fell examines the weak Frobenius properties of group $G$ in \cite{fell1964weak}. Later, Greenleaf\cite{greenleaf1969amenable} showed that the group $G$ satisfies weak Frobenius property (WF1) if and only if $G$ is amenable.  It is shown that for any irreducible representation of an amenable group, the representation is weakly contained in the induced representation of its restriction to any subgroup. It also addresses the notion of an amenable action of a group on a locally compact space $Z$, along with the existence of a left-invariant mean on $Z$.

These theories have been further generalized to encompass broader structures such as hypergroups and measurable groupoids. The concept of amenable hypergroups is introduced in \cite{skantharajah1992amenable}, and more results related to weak containment of representation of hypergroups and amenable hypergroups can be referred to in \cite{pavel2007weak,bagheri2018kazhdan}. The concept of amenable measured groupoids was introduced by Renault \cite{renault1997fourier}. Further results concerning this notion, including various equivalent characterizations of amenability—such as the existence of invariant means, Reiter conditions, and the weak containment of the trivial representation in the regular representation—are presented in \cite{MR1799683}.

In this paper, we investigate certain properties of the inner tensor product of continuous unitary representations of a locally compact groupoid $G$ and weak containment of representations of $G$. In \cite{sridharan2025induced}, we introduced the notion of an induced representation $(\operatorname{ind}_{H}^{G}(\pi),\mu)$ associated with a continuous unitary representation $\pi$ of a closed wide subgroupoid $H$, under the assumption of a full equivariant measure system $\mu=\{\mu^{u}\}_{u\in G^{0}}$ the quotient space $G/H$. The construction of induced representations via Rieffel's induction machinery has already been developed in \cite[Section 5.3]{williams2019tool}. That approach, however, relies on Renault's Disintegration Theorem and consequently requires measurable Hilbert bundles together with a quasi-invariant measure on the unit space $G^{0}$. In contrast, our approach is entirely topological: we work with continuous Hilbert bundles over $G^{0}$
 and do not take any  quasi-invariant measure on the unit space. Our development is based on the theory of continuous unitary representations of groupoids on continuous Hilbert bundles established in \cite{bos2011continuous}.
 
 We define weak containment of unitary representations of groupoids and show that weak containment is preserved by the corresponding induced representation. Similar to groups, we define the topological invariant mean on $G/H$ and prove its existence when $G$ is amenable. Using these results we prove that amenable groupoid satisfies weak Frobenius property. Unlike groups, each of these results depends on the equivariant measure system $\mu$ on $G/H$ and the Haar system $\lambda$ on $G$. Subsequently, we establish an independent result concerning the containment of the representation of a  clopen wide subgroupoid $H$ within the restriction of its induced representation. An analogous result is also demonstrated when $H$ is a closed wide transitive subgroupoid of a compact transitive groupoid $G$. These results are analogous to classical theorems concerning the representations of open subgroups of locally compact groups and closed subgroups of compact groups.

 The structure of the paper is as follows. In Section $2$, we present the necessary background on groupoids, amenable groupoids, and their representations. Section $3$ introduces and explores weak containment and the inner tensor product of representations. We establish that for a closed wide subgroupoid $H$, and arbitrary representations $\rho$ of $G$ and $\pi$ of $H$, the equality  $\rho \otimes (\operatorname{ind}_{H}^{G}(\pi),\mu)= (\operatorname{ind}_{H}^{G}(\rho_{|_{H}}\otimes \pi),\mu)$ holds. Furthermore, we show that weak containment is preserved under both induction and the inner tensor product of representations. In Section $4$, we introduce the notion of topological invariant mean on $G/H$ and define a sequence of $C_{c}(G/H)$ functions with certain properties as topological approximate invariant density on $G/H$. We demonstrate that the existence of a topological invariant mean on $G/H$ is equivalent to the existence of a topological approximate invariant density on $G/H$.  Section $5$ examines amenability in terms of weak containment of representations. In particular, we prove that if $G$ is amenable then $G$ satisfies a weak Frobenius property analogous to that for groups. The final section addresses the restriction of induced representations of groupoids. We show that for a clopen wide subgroupoid $H \subseteq G$ with discrete unit space, and a continuous unitary representation $\pi$ of $H$, $\pi$ forms a subrepresentation of the restriction of induced representation $(\operatorname{ind}_{H}^{G}(\pi),\mu)$ to $H$, where $\mu=\{\mu^{u}\}_{u\in G^{0}}$ is the full equivariant measure system of counting measures.  Similarly, we show that continuous unitary representation $\pi$ of a closed wide transitive subgroupoid $H$ of a compact transitive groupoid $G$, forms a subrepresentation of the restriction of induced representation $(\operatorname{ind}_{H}^{G}(\pi),\mu)_{|_{H}}$ where $\mu$ is a full equivariant measure system.  These results enable us to extend the Fourier–Stieltjes algebra of $H$ to $G$ in both cases.

 \section{Preliminaries}
Let's start by recalling some basics on groupoids.

 A groupoid is a set $G$ equipped with a partial product map $G^{2} \to G:(x,y) \to xy$, where $G^{2}\subseteq G\times G$ denotes the set of composable pairs, and an inverse map  $G \to G: x\to x^{-1}$,  satisfying the following conditions:
    \begin{enumerate}[(i)]
        \item $(x^{-1})^{-1}=x$,
        \item $(x,y),(y,z)\in G^{2}$ implies $(xy,z),(x,yz)\in G^{2}$ and $(xy)z=x(yz)$,
        \item $(x^{-1},x)\in G^{2}$ and if $(x,y)\in G^{2}$, then $x^{-1}(xy)=y$,
        \item $(x,x^{-1})\in G^{2}$ and if $(z,x)\in G^{2}$, then $(zx)x^{-1}=z.$
    \end{enumerate}

    A topological groupoid consists of a groupoid $G$ equipped with a topology that is compatible with its algebraic structure, in the sense that:
    \begin{enumerate}[(i)]
        \item the inverse map $G\to G: x\to x^{-1}$ is continuous,
    
        \item the product map $G^{2}\to G: (x,y)\to xy$ is continuous where $G^{2}$ is endowed with the subspace topology inherited from $G \times G$. 
    \end{enumerate}
     We restrict our attention to second-countable, locally compact, Hausdorff groupoids. The unit space $G^{0}\subseteq G$, equipped with the subspace topology, is itself a locally compact Hausdorff space. Moreover, the range and domain maps ( denoted $r,d:G \to G^{0}$ respectively) are continuous.

      Let $G$ be a topological groupoid and  $X$  a locally compact Hausdorff space equipped with a continuous map $r_{X}: X \to G^{0}$, referred to as the moment map.  A left action of $G$ on $X$ is defined by a continuous map $(\gamma,x) \to \gamma\cdot x$ from  the set $G*X= \{(\gamma,x) \in G\times X : s(\gamma) = r_{X}(x) \}$ to X satisfying the following conditions:
   \begin{enumerate}[(i)]
       \item $r_{X}(x) \cdot x= x~ \text{for all}~ x \in X$ and
       \item if $(\gamma,\eta) \in  G^{2}$ and $(\eta,x) \in  G*X$, then $(\gamma,\eta \cdot x) \in  G*X$ and $\gamma \eta \cdot x=\gamma \cdot (\eta \cdot x)$.
   \end{enumerate} 
   In this setting, $X$ is called left $G$-space.
   
   An important example is the quotient space $G/H$, where $H$ is a closed wide subgroupoid of $G$. The space $G/H$ is  the set of equivalence class obtained from the equivalence relation $\sim$ as follows: $l \sim g \Leftrightarrow r(l)=r(g)~ \text{and}~ l^{-1}g\in H$ . The space $G/H$ inherits a natural left $G$-space structure with  action $g\cdot xH=(gx)H$.
 Endowed with the quotient topology and quotient map $q_{H}$, it becomes a locally compact Hausdorff space. The associated moment map $r_{G^{0}}: G/H \to G^{0}, r_{G^{0}}(gH)=r(g)$ is both open and continuous.
 
Let $Y$ and $X$ be left $G$-spaces, and suppose $\pi: Y \to X$ is a $G$-equivariant map, meaning that for all $g \in G$ and $y \in Y$, $\pi(g \cdot y) = g \cdot \pi(y)$.
A fully equivariant $\pi$-system $\beta$ is a family of measures $\{ \beta^x : x \in X \}$ on $Y$ satisfying the following conditions:

\begin{enumerate}[(i)]
    \item $\operatorname{supp}(\beta^{x})= \pi^{-1}(x)$
    \item $x\to \beta(f)(x)= \int_{Y} f(y)d\beta^{x}(y)$ is continuous for $f\in C_{c}(Y)$
    \item For $f\in C_{c}(Y) , (\gamma,x)\in G*X$,
    $$ \int_{Y} f(\gamma \cdot y)d\beta^{x}(y)=\int_{Y} f(y)d\beta^{\gamma \cdot x}(y)$$
\end{enumerate} 

Notable examples of fully equivariant $\pi$-systems include the left Haar system on a groupoid $G$ with $\pi = r$, and the case $Y = G/H$, $X = G^0$, with $\pi = r_{G^0}$. We assume the existence of a fully equivariant $r_{G^0}$-system $\mu=\{\mu^{u}\}_{u\in G^{0}}$ on $G/H$ throughout. According to \cite[Proposition 2.4]{renault2006groupoid}, if $G$ is a locally compact groupoid equipped with a left Haar system, then the range and domain maps are open maps. For further details on groupoids and left $ G$-space, the reader is referred to \cite{renault2006groupoid, paterson2012groupoids, williams2019tool}.

Now we look into some basics of amenable locally compact groupoids.

Let $X$ and $Y$ be $G$-spaces and let $\pi:Y \to X$ be a continuous $G$- equivariant surjection. An \textit{approximate invariant continuous mean} for $\pi$ is a net $\{m_{i}\}$ of $\pi$-systems $m_{i}=\{m_i^x\}_{x\in X}$ of probability measures such that the total variation norm $\| \gamma^{-1}\cdot m^{x}_{i}-m_{i}^{\gamma^{-1}\cdot x}\|_{1}\to 0$ uniformly on compact subsets of $X*G=\{(x,\gamma): r_{X}(x)=r(\gamma)\}$. We say that a continuous $G$- equivariant surjection $\pi:Y \to X$ is \textit{amenable} if it admits an approximate invariant continuous mean.

\textbf{Amenable Groupoid}: A second countable, locally compact groupoid $G$ is amenable if the range map $r:G \to G^{0}$ is amenable. A locally compact groupoid with a Haarsystem $\lambda$ is \textit{topologically amenable}, as defined in \cite{williams2019tool}, if there is a net $\{f_{i}\} \subset C_{c}(G)$ such that 
\begin{enumerate}[(i)]
    \item the function $u \to \lambda(|f_{i}|^{2})(u)=\int_{G}|f_{i}(\gamma)|^{2}d\lambda^{u}(\gamma)$ are uniformly bounded  and
    \item the functions $\gamma\to f_{i}* f^{*}_{i}(\gamma)$ converge to a constant function $1$ uniformly on compact sets in $G$, where $f^{*}(\gamma)=\overline{f(\gamma^{-1})}$.
\end{enumerate}
The proposition \cite[Proposition $2.2.13$]{MR1799683} shows that the two are equivalent on a locally compact groupoid.

We now define continuous representations of groupoids over continuous fields of Hilbert spaces.

A continuous field of Hilbert spaces over $G^{0}$ consists of a family $\{\mathcal{H}_{u}\}_{u\in G^{0}}$ of Hilbert spaces together with a set of vector fields $\Gamma \subset \prod_{u\in G^{0}}\mathcal{H}_{u}$,  satisfying the following conditions:
\begin{enumerate}[(i)]
    \item $\Gamma$ is a complex linear subspace of $\prod_{u\in G^{0}}\mathcal{H}_u$.
    \item For every $u\in G^{0}$, the set $\{\xi(u): \xi \in \Gamma\}$ is dense in $\mathcal{H}_{u}$.
    \item For every $\xi \in \Gamma$, the function $u \to \|\xi(u)\|$ is continuous.
    \item Let $\xi \in \prod_{u\in G^{0}}\mathcal{H}_{u}$ be a vector field; if for every $u\in G^{0}$ and every $\varepsilon >0$, there exists an $\xi'\in \Gamma$ such that $\|\xi(s) -\xi'(s)\| < \varepsilon$ on a neighbourhood of $u$, then $\xi \in \Gamma$.
\end{enumerate}
Given such a field, we define a topology on the disjoint union $\mathcal{H}=\sqcup_{u\in G^{0}}\mathcal{H}_{u}$, generated by the sets of the form 
\[U(V,\xi,\varepsilon)=\{h\in \mathcal{H}:\|h-\xi(p(h))\|<\varepsilon,\xi\in \Gamma,p(h)\in V\}.\]
 where $V$ is an open set in $G^{0}$, $\varepsilon> 0$, and $p: \mathcal{H}\to G^{0}$ is the projection of the total  space $\mathcal{H}$ to base space $G^{0}$ such that fiber $p^{-1}(u) = \mathcal{H}_{u},u \in G^{0}$. Under this topology, the map $p$ is continuous, open, and surjective.

Endowed with the above topology,     $(\mathcal{H},\Gamma)$  becomes a continuous Hilbert bundle, and $\Gamma$ corresponds to the space of continuous sections.

The space of continuous sections vanishing at infinity is denoted by $C_{0}(G^{0},\mathcal{H})$. Since $G^{0}$ is locally compact, the space $C_{0}(G^{0},\mathcal{H})$ is fiberwise dense in $\Gamma$ and forms a Banach space with respect to the supremum norm. It also forms a Hilbert module over $C_{0}(G^{0})$.

Next, we define the notion of continuous unitary representation of groupoids as given in \cite[Section $3.1$]{bos2011continuous}.
A  continuous  representation of groupoid $G$ is a pair $(\mathcal{H}^{\pi},\pi)$, where $\mathcal{H}^{\pi}=\{\mathcal{H}^{\pi}_{u}\}_{u\in G^{0}}$ is a continuous Hilbert bundle over $G^{0}$ such that  
\begin{enumerate}[(i)]
    \item  $\pi(x) \in \mathcal{B}\left(\mathcal{H}^{\pi}_{d(x)}, \mathcal{H}^{\pi}_{r(x)}\right)$ is a unitary operator, for each $x \in G$,

\item $\pi(u)$ is the identity map on $\mathcal{H}^{\pi}_{u}$ for all $u \in G^{0}$,

\item $\pi(x) \pi(y)=\pi(x y)$ for all $(x, y) \in G^{2}$,

\item $\pi(x)^{-1}=\pi\left(x^{-1}\right)$ for all $x \in G$,

\item $\pi_{\xi,\eta}(x)= \langle \pi(x) \xi(d(x)),\eta(r(x))\rangle$ is continuous for every $\eta,\xi \in C_{0}(G^{0},\mathcal{H}^{\pi})$.
\end{enumerate}
The reader is referred to \cite{bos2007groupoids,bos2011continuous} for more details. By a slight abuse of notation, we shall sometimes write   $\langle \pi(x) \xi,\eta\rangle$ in place of $ \langle \pi(x) \xi(d(x)),\eta(r(x))\rangle$.

The following is the definition of Fourier-Steiltjes algebra given in \cite{paterson2004fourier}.
\begin{defi}
The \emph{Fourier--Stieltjes algebra} of a locally compact groupoid
$G$ is defined as
\[
B(G)
:=
\left\{
\pi_{\xi,\eta}
:
\mathcal{H}^{\pi}\ \text{is a } G\text{-Hilbert bundle},
\;
\xi,\eta\in C_{b}(G^{0},\mathcal{H}^{\pi})
\right\},
\]
equipped with pointwise operations. The norm is defined by
\[
\|\phi\|_{B(G)}
:=
\inf_{\pi=\pi_{\xi,\eta}}
\|\xi\|
\|\eta\|
\]
where the infimum is taken over all coefficient representations of
$\phi=\pi_{\xi,\eta}$.
\end{defi}

A function $\phi\in C(G)$ is said 
to be \emph{positive definite} if for all
$u\in G^{0}$ and all $f\in C_{c}(G)$ we have
\begin{equation*}
\iint
\phi(y^{-1}x)\,
f(y)\,
\overline{f(x)}
\,d\lambda^{u}(x)\,
d\lambda^{u}(y)
\ge 0.
\end{equation*}
It is also proved in \cite[Theorem $1$]{paterson2004fourier}, that $\phi$ is positive definite if and only if $\phi$ is a coefficient of the form
$\pi_{\xi,\xi}$ for some $G$-Hilbert bundle $\mathcal{H}^{\pi}$ and $\xi \in  C_{b}(G^{0},\mathcal{H}^{\pi}) $. A positive definite function $\phi$ is said to be associated with the representation $\pi$ if there exists a bounded continuous section $\xi \in C_{b}(G^{0},\mathcal{H}^{\pi})$ satisfying $\phi(x)=\pi_{\xi,\xi}(x).$

\section{Weak Containment and Inner Tensor Product of Representation}
Before moving to the main part, let us recall some important results on induced representation given in \cite{sridharan2025induced}.

 Suppose $(\mathcal{H}^{\sigma},\sigma)$ is a continuous unitary representation of a closed wide subgroupoid $H$ with Haar system $\lambda_{H}$.  Let $C(G,\mathcal {H}^{\sigma})$ denotes the set of  continuous function $f$ such that $f(x) \in \mathcal{H}^{\sigma}_{d(x)}$ and its subspace $C_{c}(G,\mathcal {H}^{\sigma})$ be the functions with compact support. 
Define,
\begin{align*}
  \mathcal{F}_{0}^{\sigma}(G)=\Bigg\{f\in C(G,\mathcal {H}^{\sigma}):~ & q_{H}(\operatorname{supp}(f))~ \text{is compact}~ \text{and} \\&f(xh)=\sigma(h^{-1})f(x),\text{for}~ (x,h)\in G^{2},h \in H \Bigg \}. 
\end{align*}

It is proved in \cite[Proposition $3.2$]{sridharan2025induced} that  every element of $\mathcal{F}_{0}^{\sigma}(G)$ is of the form 

    \begin{equation}
        f_{\alpha}(x)= \int_{H} \sigma(h)\alpha(xh)d\lambda_{H}^{d(x)}(h),
    \end{equation}
 for some $\alpha \in C_{c}(G,\mathcal{H}^{\sigma})$.
 $\mathcal{F}_{0}^{\sigma}(G)$ form a left pre-Hilbert $C_{0}(G^{0})$-module with innerproduct
    $$ \langle f,g\rangle(u)=\int_{G/H}\langle f(x),g(x)\rangle_{\mathcal{H}^{\sigma}_{d(x)}}d\mu^{u}(xH).$$
 Its completion $\mathcal{F}^{\sigma}(G)$ or $\mathcal{F}^{\sigma}(G,\mu)$ under the norm 
 
 $$ \|f\|= \sup_{u\in G^{0}}\sqrt{\langle f,f\rangle(u)},$$
  forms a left Hilbert $C_{0}(G^{0})$-module(see \cite[Lemma $3.3$]{sridharan2025induced}).

 The Hilbert module $\mathcal{F}^{\sigma}(G)$ can be identified with a continuous section  $\Delta=\{t_{F} \in C_{0}(G^{0},\mathcal{H}^{\operatorname{ind}_{H}^{G}(\sigma)}):t_{F}(u)= F 
 + \mathcal{N}^{u},F\in \mathcal{F}^{\sigma}(G)\}$, where $ \mathcal{N}^{u}= \{f\in \mathcal{F}^{\sigma}(G): \|f\|^{2}(u)=\langle f,f\rangle(u)=0\}$, of a continuous Hilbert bundle $\mathcal{H}^{\operatorname{ind}_{H}^{G}(\sigma)}$ whose fibers $\{\mathcal{H}^{\operatorname{ind}_{H}^{G}(\sigma)}_{u}\}_{u\in G^{0}}$ can be identified with the completion of $\{f_{|_{G^{u}}}: f\in \mathcal{F}_{0}^{\sigma}(G)\}_{u\in G^{0}}$ under the norm $\sqrt{\langle f,f\rangle(u)}$.

The representation $(\mathcal{H}^{\operatorname{ind}_{H}^{G}(\sigma)},\Delta,L)$ of $G$ where $$ L(x): \mathcal{H}^{\operatorname{ind}_{H}^{G}(\sigma)}_{d(x)}\to \mathcal{H}^{\operatorname{ind}_{H}^{G}(\sigma)}_{r(x)}, ~~(L(x)f)(y)= f(x^{-1}y), $$ is defined as  representation  induced by $\sigma$ and it is denoted by $\left(ind^{G}_{H}(\pi),\mu \right)$ (see \cite[Proposition $3.5$]{sridharan2025induced}).

We have already  the outer tensor product of representations of $G$ in \cite{sridharan2025induced}. We also proved the groupoid analogue of the Mackey tensor product theorem. Here, we introduce what is called the inner tensor product of representations. Subsequently, we utilize these results to establish some important theorems.

 Suppose $\pi$ and $\rho$ be two representations of $G$, the inner tensor product $\pi \otimes \rho$ is the representation of $G$ on the Hilbert bundle   $\mathcal{H}^{\pi} \otimes^{'} \mathcal{H}^{\rho}$ over $G^{0}$ having Hilbert spaces $\{\mathcal{H}^{\pi}_{u}\otimes \mathcal{H}^{\rho}_{u}\}_{u\in G^{0}}$ as fibers. It is defined as, $(\pi \otimes \rho)(x)(v\otimes w)= \pi(x)v \otimes \rho(x)(w) \in \mathcal{H}^{\pi}_{r(x)}\otimes \mathcal{H}^{\rho}_{r(x)} $, for $v\otimes w \in \mathcal{H}^{\pi}_{d(x)}\otimes \mathcal{H}^{\rho}_{d(x)} $. 
 
 Let $H$ be a closed wide subgroupoid of $G$ and $\mu$ an equivariant measure system of $G/H$.
 The following theorem gives a relation between induced representation and inner tensor product of representations.
 \begin{thm}
     Let $H$ be a closed wide subgroupoid of $G$. Then for any arbitrary representation $\rho$ of $G$ and $\pi$ of $H$,
     $$\rho \otimes (\operatorname{ind}_{H}^{G}(\pi),\mu)= (\operatorname{ind}_{H}^{G}(\rho_{|_{H}}\otimes \pi),\mu).$$
 \end{thm}
 \begin{proof}
     For $u \in G^{0}$, Define $T_{u}: \mathcal{H}^{\rho}_{u}\otimes \mathcal{H}^{\operatorname{ind}_{H}^{G}(\pi)}_{u} \to \mathcal{H}^{\operatorname{ind}_{H}^{G}(\rho_{|_{H}}\otimes \pi)}_{u}$ such that 
     $$ T_{u}(k \otimes f)(x)= \rho(x^{-1})k \otimes f(x),~~ x\in G^{u}.$$
The map $T:C_{0}(G^{0},\mathcal{H}^{\rho} \otimes^{'} \mathcal{H}^{\operatorname{ind}_{H}^{G}(\pi)}) \to \mathcal{F}^{\rho_{|_{H}}\otimes \pi}(G)$ preserves the $C_{0}(G^{0})$-valued innerproduct.
      For $g \in C_{c}(G),~ \xi \in C_{0}(G^{0},\mathcal{H}^{\rho}), w \in C_{0}(G^{0},\mathcal{H}^{\pi})$, let $\mathcal{E}(g,w)$ be as defined in \cite[Lemma 3.7]{sridharan2025induced}. Then,
     $$ T_{r(y)}[(\xi \otimes \mathcal{E}(g,w)](y)= \rho(y^{-1})\xi(r(y)) \otimes \mathcal{E}(g,w)(y)=f_{\alpha}(y),$$
     where $f_{\alpha}$ is the function defined in $(1)$ with $\alpha(x)= g(x)(\rho(x^{-1})\xi(r(x))\otimes w(d(x)).$
     
      Also, for $g\in C_{c}(G),~ \xi \otimes w \in C_{0}(G^{0},\mathcal{H}^{\rho} \otimes^{'} \mathcal{H}^{\pi})$, 
      $$ \mathcal{E}(g, \xi \otimes w )(y)= f_{\beta}(y), $$
      where $\beta(y)= \rho(y^{-1})k'(y)\otimes w(d(y))$ with $k'(y)= g(y)\rho(y)\xi(d(y))$. By continuity of $k'$, we can easily observe that $f_{\beta}$ is in the closure of the range of $T$. Hence, $T$ is surjective. Moreover, $T$ intertwines with the corresponding representations.
 \end{proof}
Now we define weak containment of a continuous representation of groupoid $G$. We also show that the weak containment is preserved in the inner tensor product of two representations.
\begin{defi}
Let $\pi$ and $\rho$ be two continuous representations of the groupoid $G$. Then, we say $\pi$ is weakly contained in $\rho$ denoted as $\pi \prec \rho$ if  for every  positive definite function $\phi$ associated with $\pi$, compact subset $K$ of $G$, and $\varepsilon >0$, there exist $\eta_{1},..., \eta_{n} \in C_{b}(G^{0}, \mathcal{H}^{\rho})$ such that 
$$  \left|\phi(x)- \sum_{i=1}^{n} \langle \rho(x)\eta_{i}(d(x)), \eta_{i}(r(x))\rangle\right| < \varepsilon,~~\text{for every}~~x\in K.$$
\end{defi}

The following proposition shows that weak containment is preserved in the inner tensor product of two representations.
\begin{lem}
    Let $\pi_{1},\pi_{2},\rho_{1},\rho_{2}$ be representations of groupoid $G$ such that $\pi_{1}\prec\rho_{1}$ and $\pi_{2}\prec \rho_{2}$, then $\pi_{1}\otimes \pi_{2} \prec \rho_{1} \otimes \rho_{2}$.
\end{lem}
\begin{proof}
    Let $\phi(x)= \langle \pi_{1}\otimes \pi_{2}(x) \xi \otimes \eta(d(x)),\xi\otimes \eta(r(x))\rangle$. For a compact set $K$ and $\varepsilon >0$, let $M=\sup_{x\in K}|\langle \pi_{2}(x)\eta(d(x)),\eta(r(x))\rangle|$, then there exist $t_{i} \in C_{b}(G,\mathcal{H}^{\rho_{1}}), 1\leq i\leq n$ and $s_{j} \in C_{b}(G,\mathcal{H}^{\rho_{2}}), 1\leq j \leq m$ such that 
    \begin{align*}
        & \left|\langle \pi_{1}(x)\xi(d(x)),\xi(r(x))\rangle - \sum_{i=1}^{n}\langle \rho_{1}(x)t_{i}(d(x)),t_{i}(r(x))\rangle \right|< \frac{\varepsilon}{2M}\\
        &\left|\langle \pi_{2}(x)\eta(d(x)),\eta(r(x))\rangle - \sum_{i=1}^{m}\langle \rho_{2}(x)s_{i}(d(x)),s_{i}(r(x))\rangle \right|< \frac{\varepsilon}{2N}
    \end{align*}
    
    where $N= \sup_{x\in K}\left|\sum_{1=1}^{n}\langle \rho_{1}(x)t_{i}(d(x)),t_{i}(r(x))\rangle \right|$.
    Then,
    $$ \left|\phi(x) - \sum_{i,j}\langle (\rho_{1}\otimes \rho_{2})(x)t_{i}\otimes s_{j}(d(x)), t_{i}\otimes s_{j}(r(x))\rangle \right| < \varepsilon,~ x\in K.\qedhere  $$  
\end{proof}
\begin{lem}
    Let $G$ be a locally compact groupoid , $H$ a closed wide subgroupoid and $\pi , \rho$ be two unitary representations of groupoid $H$ such that $\pi \prec \rho$, then $(\operatorname{ind}_{H}^{G}(\pi),\mu)\prec (\operatorname{ind}_{H}^{G}(\rho),\mu)$
\end{lem}
\begin{proof}
 Let $k(x) = \mathcal{E}(f,t)(x)=\int_{H} f(xh)\pi(h)t(d(h))d\lambda_{H}^{d(x)}(h), f\in C_{c}(G),~ t\in C_{0}(G^{0},\mathcal{H}^{\pi})$. Elements of this form is total in the Hilbert module $\mathcal{F}^{\pi}(G)$ by \cite[Lemma $3.6$]{sridharan2025induced}.
 \begin{align*}
      \langle \operatorname{ind}_{H}^{G}&(\pi)(x)k,k \rangle=\int_{G/H}\Big\langle \mathcal{E}(f,t)(x^{-1}y),\mathcal{E}(f,t)(y)\Big\rangle d\mu^{r(x)(yH)}\\
       &=\iint\limits_{H} f(x^{-1}yh)\bar{f}(yh')\Big\langle  t(d(h)),\pi(h^{-1}h')(t(d(h'))\Big\rangle d\lambda_{H}^{d(y)}(h)d\lambda_{H}^{d(y)}(h').   
     \end{align*} 

     Since $\pi$ is weakly contained in  $\rho$, there exist $s_{i}\in  C_{0}(G^{0},\mathcal{H}^{\rho})$ such that the function  $\Big\langle  t(d(h)),\pi(h^{-1}h')(t(d(h'))\Big\rangle$ can be approximated on  required compact sets by $\sum_{i=1}^{n} \Big\langle  s_{i}(d(h)),\rho(h^{-1}h')(s_{i}(d(h'))\Big\rangle $. Hence from the above, one can easily verify that  $\langle \operatorname{ind}_{H}^{G}(\pi)(x)k,k \rangle$ can be approximated by $\sum_{i=1}^{n} \langle \operatorname{ind}_{H}^{G}(\rho)(x)k_{i},k_{i} \rangle $, where $k_{i}= \mathcal{E}(f,s_{i}) $, on  required compact sets.
\end{proof}

\section{Topological Invariant Mean on $G/H$}
 Here, we introduce the topological invariant mean on $G/H$ with respect to a full equivariant measure system $\mu$ and show that when $G$ is amenable $G/H$ has a topological invariant mean.
 
 Let $G/H$ be a $G$-space with a full equivariant measure system $\mu$, and $\mathcal{E}=C_{0}(G^{0},L^{1}(G/H,\mu))$ be the completion of $C_{c}(G/H)$ with respect to the norm
 $$\|f\|_{\mu,1}= \sup_{u\in G^{0}}\int_{G/H} |f(\omega)|d\mu^{u}(\omega),$$
 then we can define a map from $C_{c}(G)$ to $B(C_{0}(G^{0},L^{1}(G/H,\mu)))$ as 
 $$L(f)\phi(xH)= \int_{G} f(y) \phi(y^{-1}xH)d\lambda^{r(x)}(y), ~\phi\in C_{c}(G/H).$$
 The boundedness can be easily verified and thus the map extends to $C_{0}(G^{0},L^{1}(G/H,\mu))$
 
Also using \cite[Proposition $1.1.5$]{MR1799683}, the space of $\mu$-bounded complex radon measures on $G/H, M(G/H,\mu)$, forms the dual space of $\mathcal{E}$. For $f\in C_{c}(G),m\in \mathcal{E}^{**}$, $f*m$ be the double transposition of $L(f)$ acting on $\mathcal{E}^{**}$. Thus if $g\in C_{c}(G/H)$,
$$\langle f*m_{g},\nu\rangle=\nu(L(f)g)= \langle m_{L(f)g},\nu\rangle,$$
where $g \to m_{g}$ be the standard embedding of $\mathcal{E}$ into $\mathcal{E}^{**}$. We also define 
$$ \langle k\cdot \nu,g\rangle = \nu(kg)~~ \text{and}~~ \langle k\cdot m,\nu \rangle = m(k\cdot \nu)$$
where $k\in C_{0}(G^{0}), m\in \mathcal{E}^{**},\nu \in \mathcal{E}^{*}$ and $kg(\omega)=k(r_{G^{0}}(\omega))g(\omega)$ 
\begin{defi}
    Let $G/H$ be a $G$-space with full equivariant measure system $\mu$. A  topological invariant mean on $G/H$ with respect to $\mu$ is an element $m$ in $\mathcal{E}^{**}$ such that 
    \begin{enumerate}
        \item $m(\nu)\geq 0 $ if $\nu\geq0$,
        \item for any probability measure $\eta$ on $G^{0}$, $m(\eta\circ\mu)=1$ and
        \item for any $f\in C_{c}(G)$, $f*m=\lambda(f)\cdot m$ as a functional on $\mathcal{E}^{*}$.
    \end{enumerate}
\end{defi}
We call a sequence of functions $\{f_{i}\}_{i\in \mathbb{N}}$ in $G/H$ with full equivariant measure system $\mu$ a \textit{topological approximate invariant density} on $G/H$ if 
\begin{enumerate}[(i)]
    \item $\int_{G/H} f_{i}(xH) d\mu^{u}(xH) \leq 1 ~~ \text{for all}~~ i~~\text{and}~u$,
    \item $\operatorname{lim}_{i} \int_{G/H} f_{i}(xH) d\mu^{u}(xH)=1$  uniformly on compact subsets of $G^{0}$, and
    \item $\int_{G/H} |f_{i}(y^{-1}xH)- f_{i}(xH)| d\mu^{r(y)}(xH)$ tends to zero uniformly on compact subsets of $G$.
\end{enumerate}
Similarly, we call those sequences of functions satisfying  \cite[Proposition $2.2.13(ii)$]{MR1799683} as \textit{topological approximate invariant density} on $G$ with Haar system $\lambda$ . 

The following proposition shows that the existence of a topological invariant mean on $G/H$ is equivalent to the existence of a topological approximate invariant density on $G/H$. With minor modifications, the proof closely resembles the classical case of a locally compact group (see \cite[Theorem G.3.1]{bekka2007kazhdan}) as well as the groupoid case (see \cite[Theorem 2.14]{renault2015topological}.
\begin{thm}
    Let $G$ be a topological groupoid with Haar system $\lambda$ and $H$ a closed wide subgroupoid; then $G/H$ admits a topological invariant mean iff $G/H$ admits a topological approximate invariant density.
\end{thm}
  \begin{proof}
      Let $\{g_{i}\}_{i\in \mathbb{N}}$, be the topological approximate invariant density on $G/H$. Then, $\|m_{i}\|\leq 1$, where $m_{i}=m_{g_{i}} \in \mathcal{E}^{**}$. Let $m$ be the weak-$^{*}$ cluster point of $m_{i}$. We show that $m$ satisfies every condition of Definition $4.1$. Since $m_{g_{i}}(\nu)\geq0$ for all $g_{i}$, $m(\nu)\geq0$ whenever $\nu\geq0$. Let $\eta$ be a probability measure on $G^{0}$, then by dominated convergence theorem,
      $$\lim_{i} m_{g_{i}}(\eta \circ \mu)= \lim_{i}\int_{G^{0}}\int_{G/H}g_{i}(\omega)d\mu^{u}(\omega)d\eta(u)=1.$$ Hence the above satisfies for the subnet $\{m_{i_{k}}\}$ that converge to $m$. 
 
      To prove $(3)$ in Definition $5.1$, consider $f \in C_{c}(G)$. Using \cite[Remark $1.1.6$]{MR1799683}, any $\nu \in \mathcal{E}^{*}$ can be written as $\nu=\varphi(\mu_{0}\circ\mu)$, where $\mu_{0}$ is a bounded positive measure of $G^{0}$ and for $\mu_{0}$-a.e. $u$, the $\mu^{u}$-essential supremum of $|\varphi|$ is $1$. Now,
      \begin{align*}
          \langle f*m_{i}-&\lambda(f)\cdot m_{i},\nu\rangle =\langle \nu,L(f)g_{i}-\lambda(f)g_{i}\rangle\\
          &= \int_{G^{0}}\int_{G/H}\varphi(\omega)[(L(f)g_{i}(\omega)-\lambda(f)(u)g_{i}(\omega)]d\mu^{u}(\omega)\mu_{0}(u)\\
          &= \int_{G^{0}}\int_{G/H}\int_{G} \varphi(\omega)f(y)[g_{i}(y^{-1}\cdot \omega)-g_{i}(\omega)]d\lambda^{u}(y)d\mu^{u}(\omega)\mu_{0}(u)\\
          &=\int_{G^{0}}\int_{G} f(y)\left(\int_{G/H} \varphi(\omega)[g_{i}(y^{-1}\cdot \omega)-g_{i}(\omega)]d\mu^{r(y)}(\omega)\right)d\lambda^{u}(y)\mu_{0}(u)
      \end{align*}
      With the property of the sequence $\{g_{i}\}$,
      $$\left|\int_{G/H} \varphi(\omega)[g_{i}(y^{-1}\cdot \omega)-g_{i}(\omega)]d\mu^{r(y)}(\omega)\right| \to 0,  $$ for every $y \in G$ and it is bounded by $2\|\varphi \|_{\infty}.$ Hence, using dominated convergence theorem applied on $L^{1}(\mu_{0}\circ \mu)$ we can see that 
      $$ \langle f*m_{i}-\lambda(f)\cdot m_{i},\nu \rangle \to 0~~ \text{for all}~~ \nu.$$ 
      The above holds for any subnet $m_{i_{k}}$ that converges to $m$ and hence $m$ is a topological invariant mean on $G/H$.

      Now, let $m$ be a topological invariant mean on $G/H$. Then by \cite[Lemma $1.2.7$]{MR1799683}, there exist a net of measures  $\{m_{g_{i}}\}_{i\in I}$, with $g_{i}\geq 0$ and $\|g_{i}\|_{{\mu},1}\leq 1 $ converging to $m$ under weak-$^{*}$ topology. It can be easily verified that for any probability measure $\eta$ on $G^{0}$ and $\varphi \in L^{\infty}(\eta)$,
      \begin{equation}\int_{G^{0}}\varphi(u)\mu(g_{i})(u) d\eta(u) \to \int_{G^{0}} \varphi(u)d\eta(u) \end{equation} and for every $f\in C_{c}(G)$, 
      \begin{equation} L(f)g_{i}-\lambda(f)g_{i} \to 0 \end{equation} in weak topology on $\mathcal{E}$.

      Let $\mathcal{E}_{0}=C_{b}(G^{0})$ be under strict topology and for each $f\in C_{c}(G)$, $\mathcal{E}_{f}$ be $\mathcal{E}$ with norm topology. The strict dual of $C_{b}(G^{0})$ can be realized as $M(G^{0})$ and let $\mathcal{E}_{0,w}$ be $C_{b}(G^{0})$ with the weak topology. Denote $\mathcal{E}_{f}$ under weak topology as $\mathcal{E}_{f,w}$. By \cite[Proposition $A.5$]{williams2019tool}, the the topological vector space $\mathcal{F}=\mathcal{E}_{0} \times \prod\limits_{f\in C_{c}(G)} \mathcal{E}_{f}$ with weak topology, denoted $\mathcal{F}_{w}$, is $\mathcal{F}_{w}=\mathcal{E}_{0,w} \times \prod\limits_{f\in C_{c}(G)} \mathcal{E}_{f,w}$.
       Since every closed convex set of a topological vector space is weakly closed, both $\mathcal{F}$ and $\mathcal{F}_{w}$ have the same closed convex sets. Hence using $(1)$ and $(2)$, $(1,(0)_{f\in C_{c}(G)})$ belongs to the closure of the convex set $$C = \{(\mu(g),(L(f)g-\lambda(f)g)_{f\in C_{c}(G)}): g\in C_{c}^{+}(G/H), \|g\|_{\mu,1}\leq 1 \}$$ in $\mathcal{F}$. So there exists a net again denoted as $\{g_{i}\}$, such that $\mu(g_{i}) \to 1$ strictly in $C_{b}(G^{0})$ and such that for all $f\in C_{c}(G)$,
       $$\|L(f)g_{i}-\lambda(f)g_{i}\|_{\mu,1} \to 0. $$ So, $\mu(g_{i}) \to 1$ uniformly on compact subsets of $G^{0}$ and $$ \int_{G/H}\left|\int_{G} f(y)[g_{i}(y^{-1}\cdot\omega)-g_{i}(\omega)]d\lambda^{u}(y)\right|d\mu^{u}(\omega) \to 0$$ uniformly on $G^{0}$. For $f\in C_{c}(G)$, one can easily verify that  
       $$\int_{G/H}\left|\int_{G} f(\xi^{-1}\gamma)[g_{i}(\gamma^{-1}\cdot\omega)-g_{i}(\omega)]d\lambda^{r(\xi)}(\gamma)\right|d\mu^{r(\xi)}(\omega) \to 0$$ uniformly on compact subsets of $G$.

       Now take compact sets $K \subseteq G$ and $L \subseteq G^{0}$ such that $r(K) \subseteq L$ and $K \subseteq G^{L}_{L}$.  Take $f\in C_{c}^{+}(G), \|f\|_{\lambda,1}\leq 1$ such that $\lambda(f)\equiv 1$ on $L$.
       There exist $i_{0}$ such that for all $i\geq i_{0}$
       $$\mu(g_{i})(u)\geq 1-\varepsilon,~~~ \forall u\in d(\operatorname{supp}(f)).$$
       Also,
       $$ \int_{G/H}\left|\int_{G}f(\xi^{-1}\gamma)[g_{i}(\gamma^{-1}\omega)-g_{i}(\omega)]d\lambda^{r(\xi)}(\gamma)\right|d\mu^{r(\xi)}(\omega)< \frac{\varepsilon}{2},~~ \forall \xi \in K.$$

       Fixing some $i \geq i_{0}$, define
       $$ g(\omega)=\int _{G} f(\gamma)g_{i}(\gamma^{-1}\omega)d\lambda^{r_{G^{0}}(\omega)}(\gamma).$$
       Note that $\|g\|_{\mu,1}\leq 1$ and $\mu(g)(u)\geq 1-\varepsilon$ for $u\in L$.
       For $\xi \in K$, we can show that 
       \begin{align*}
           g(\xi^{-1}&\omega)-g(\omega)\\&= \int_{G}f(\xi^{-1}\gamma)[g_{i}(\gamma^{-1}\omega)-g_{i}(\omega)]d\lambda^{r(\xi)}(\gamma)- \int_{G}f(\gamma)[g_{i}(\gamma^{-1}\omega)-g_{i}(\omega)]d\lambda^{r(\xi)}(\gamma)
       \end{align*}
       By the above property,
       $$ \int_{G/H}\left|g(\xi^{-1}\omega)-g(\omega)\right|d\mu^{r(\xi)}(\omega)< \varepsilon,~~ \forall \xi \in K.$$
       Since $G$ and $G^{0}$ are second countable, they are exhaustible by compact sets, thus helping to find a sequence of $g_{i}$'s with the required property.
   \end{proof}  

We now present a sufficient condition on the groupoids $G$ and $H$ that ensures the existence of a topological invariant mean on $G/H$.

\begin{prop}
    Let $G$ be an amenable locally compact groupoid with Haarsystem $\lambda$, $H$ be a closed wide subgroupoid of $G$ wth Haar system $\lambda_{H}$, and $\mu$ be a full equivariant measure system on $G/H$ . Then, $G/H$ possesses a topological invariant mean.
\end{prop}
\begin{proof}

 Note that amenability is independent of the choice of Haar system of the groupoid due to \cite[Theorem $2.2.17$]{MR1799683}. So, the groupoid $G$ with another Haar system  $\lambda'$  defined as 
    $$ \int_{G} g d\lambda'^{u}= \int_{G/H} Pg d\mu^{u}, ~~ g\in C_{c}(G),$$
    where $Pg(yH)= \int_{H} g(yh)d\lambda_{H}^{d(y)}(h)$, is also amenable. Hence, by \cite[Proposition $2.2.13(ii)$]{MR1799683}, there exists a topological approximate invariant density $\{g_{i}\}$ on $G$ with  Haar system $\lambda'$.
Also, one can see that
\begin{align*}
    \int_{G/H}|Pg_{i}(x^{-1}yH)-Pg_{i}(yH)|d\mu^{r(x)}(yH)&= \int_{G/H}|P(L_{x}g_{i}-g_{i})(yH)|d\mu^{r(x)}(yH)\\
   &\leq\int_{G/H}P|(L_{x}g_{i}-g_{i})|(yH)d\mu^{r(x)}(yH)\\ &=\int_{G}|g_{i}(x^{-1}y)-g_{i}(y)|d\lambda'^{r(x)}(y).
\end{align*}Now we can conclude  that $\{Pg_{i}\}$ forms a topological approximate invariant density on $G/H$. Hence, by Theorem $4.2$, the result follows.\qedhere

\end{proof}

\begin{rem}
    Observe that the aforementioned property of $G/H$ remains valid for any full equivariant measure system $\mu$, as the amenability of the groupoid $G$ does not depend on the choice of Haar system.
\end{rem}

\section{Amenability and Induced Representation}
In this section, we prove some properties of amenable groupoids in terms of weak containment of representations.  We prove a groupoid version of Greenleaf's theorem.
\begin{defi}
    A locally compact groupoid $G$ satisfies weak Frobenius property if for every closed wide subgroupoid $H$ with $G/H$ having a full equivariant measure system $\mu$ and unitary representation $\pi$ of $G$, $\pi \prec (\operatorname{ind}_{H}^{G}(\pi_{|_{H}}),\mu)$. 
\end{defi}
The following theorem provides a relation between amenable groupoid and $weak~ Frobenious ~property$. This is a groupoid version of a portion of  the well known Greenleaf's theorem.
\begin{thm}
    Let $G$ be a locally compact amenable groupoid, then $G$ satisfies weak Frobenius property.
          
\end{thm}
\begin{proof}
    Suppose $G$ is amenable. It is enough to show that $1_{G} \prec \operatorname{ind}_{H}^{G}(1_{H})$. This is because, using Theorem $3.1$ and Lemma $3.3$, we get
    $$\pi = \pi \otimes 1_{G} \prec \pi \otimes \operatorname{ind}_{H}^{G}(1_{H}) = \operatorname{ind}_{H}^{G}(\pi_{|_{H}}\otimes 1_{H})= \operatorname{ind}_{H}^{G}(\pi_{|_{H}}) .$$ 
    
    Since $G$ is amenable, using Proposition $4.3$, $G/H$  possesses a topological invariant mean. Let $K$ be a compact set in $G$ and $\varepsilon >0$. Then by Theorem $4.2$, there exist a $\varphi \in C_{c}^{+}(G/H)$ such that $\|\varphi\|_{\mu,1}\leq 1, \int_{G/H}\varphi d\mu^{u}>1-\frac{\varepsilon}{2}$ for all $u \in r(K)$ and $$ \int_{G/H}\left|\varphi(x^{-1}yH)-\varphi(yH)\right|d\mu^{u}(yH)< \frac{\varepsilon^{2}}{4}  $$ for 
    $(x,u) \in (K \times r(K)) \cap (G* G^{0})$.

    Define, $\xi(g)= \varphi(gH)^{1/2}, \forall g \in G$. It can be easily seen that $\xi \in \mathcal{F}^{1_{H}}(G)$. Then,
    \begin{align*}
        \langle \operatorname{ind}_{H}^{G}(1_{H})(g)\xi,\xi \rangle &= \int_{G/H}\langle \operatorname{ind}_{H}^{G}(1_{H})(g)\xi(y),\xi(y) \rangle d\mu^{r(g)}(yH)\\
        &= \int_{G/H}\varphi(g^{-1}y)^{1/2}\varphi(y)^{1/2}d\mu^{r(g)}(yH).
    \end{align*}
    For $g\in K$,
\begin{align*}
       \left |\langle \operatorname{ind}_{H}^{G}(1_{H})(g)\xi,\xi \rangle-1\right| &\leq \left|\int_{G/H}(\varphi(g^{-1}y)^{1/2}-\varphi(y)^{1/2})\varphi(y)^{1/2} d\mu^{r(g)}(yH)\right|+ \frac{\varepsilon}{2}\\
        &\leq \left(\int_{G/H}\left|(\varphi(g^{-1}y)^{1/2}-\varphi(y)^{1/2})\right|^{2} d\mu^{r(g)}(yH)\right)^{1/2}+ \frac{\varepsilon}{2}\\
        &\leq \left(\int_{G/H}\left|(\varphi(g^{-1}y)-\varphi(y)\right| d\mu^{r(g)}(yH)\right)^{1/2}+ \frac{\varepsilon}{2}\\
        & < \varepsilon
    \end{align*}
    Hence we proved the required result. \qedhere
    
\end{proof}

 The following result gives a sufficient condition on a pair of representations—one of a groupoid $G$ and one of a subgroupoid $H$- under which the representation of $G$ is weakly contained in the representation induced from the representation of $H$.
 \begin{corl}
     Let $H$ be a closed wide subgroupoid of an amenable groupoid $G$  with $\mu$ being a full equivariant measure system on $G/H$, $\pi$ and $\tau$ be representations of $G$ and $H$ respectively such that $\tau \prec \pi_{|_{H}}$ . Suppose,
     \begin{enumerate}[(i)]
         \item $1_{H} \prec \bar{\tau} \otimes\tau$
         \item if $\rho$ is a representations of $G$, then $1_{G}\prec \bar{\pi}\otimes \rho$ implies $\pi \prec \rho$.
     \end{enumerate}
     Then $\pi \prec (\operatorname{ind}_{H}^{G}(\tau),\mu)$.
 \end{corl}
 \begin{proof}
     Let $\pi$ and $\tau$ be representations of $G$ and $H$ respectively such that $\tau \prec \pi_{|_{H}}$. Then by $(i)$, $1_{H} \prec \bar{\tau} \otimes\tau$ implies $1_{H}  \prec \bar{\pi}_{|_{H}}\otimes\tau  $. Since $G$ is amenable, by Theorem $5.2$, Theorem $3.1$ and Lemma $3.4$,
     $$ 1_{G} \prec (\operatorname{ind}_{H}^{G}(1_{H}),\mu) \prec (\operatorname{ind}_{H}^{G}(\bar{\pi}_{|_{H}} \otimes\tau),\mu)= \bar{\pi}\otimes (\operatorname{ind}_{H}^{G}(\tau),\mu) $$
     Hence by $(ii)$, $\pi \prec (\operatorname{ind}_{H}^{G}(\tau),\mu)$.
 \end{proof}

 \section{Restriction of Induced representation }
So far, we have discussed about weak containment of representations of groupoid. In this section, we examine how certain closed subgroupoid representations are contained within the restriction of induced representations of groupoids. First we focus on groupoids $G$ with a discrete unit space $G^{0}$, and consider  clopen wide subgroupoids $H$ . Later, we show the same result using compact transitive groupoids and their closed wide transitive subgroupoids.

Additionally, using the above result, we show the extension of an element of the Fourier-Stieltjes algebra $B(H)$  to an element of $B(G)$.  Since every element of $B(G)$ is a linear combination of elements of $P(G)$, we only need to show the extension of positive definite functions. Here, we would like to highlight a recent study \cite[Section $4$]{MR4880521}, which investigates the surjectivity of the restriction of the Fourier–Stieltjes algebra to isotropy subgroups in certain groupoids.

 \begin{thm}
     Let $H$ be a  clopen wide subgroupoid of a locally compact groupoid $G$ with discrete unit space $G^{0}$. If $\pi$ is a continuous representation of $H$ then $\pi$ is a subrepresentation of $(\operatorname{ind}_{H}^{G}(\pi),\mu)_{|_{H}}$, where $\mu=\{\mu^{u}\}_{u\in G^{0}}$ is the  full equivariant measure system  of counting measures.  
  \end{thm}
  \begin{proof}
      We can easily see that $G/H$ is a discrete space, hence any map $\gamma: G/H \to G$ such that $q_{H}\circ \gamma=I_{G/H}$ is continuous. 
     For fixed $\gamma$, define the Hilbert $C_{0}(G^{0})$-module
      $$ \ell_{\gamma}^{2}(G/H,\mathcal{H}^{\pi})=\left\{\eta:G/H \to \mathcal{H}^{\pi}: \eta(\omega)\in \mathcal{H}^{\pi}_{d(\gamma(\omega))}, f(u)=\sum_{\omega\in r_{G^{0}}^{-1}(u)}\|\eta(\omega)\|^{2} \in C_{0}(G^{0})\right\}.$$
      Here, the action of $C_{0}(G^{0})$ and innerproduct is defined as follows: $f \cdot \eta (\omega)= f(r_{G^{0}}(\omega))\eta(\omega)$, $\langle\eta ,\gamma \rangle(u)= \sum_{\omega \in r_{G^{0}}^{-1}(u)} \langle \eta(\omega),\gamma(\omega)\rangle$. 
       
      Let $W:\ell_{\gamma}^{2}(G/H,~\mathcal{H}^{\pi})\to \mathcal{F}^{\pi}(G,\mu), W(\eta)(\gamma(\omega) h)=\pi(h^{-1})\eta(\omega) $, then it is a $C_{0}(G^{0})$-valued innerproduct preserving module map. The inverse of $W$ is
  $$ [W^{-1}\xi](\omega)=\xi(\gamma(\omega)), ~\xi \in \mathcal{F}_{0}^{\pi}(G,\mu).$$
  Thus, $\ell_{\gamma}^{2}(G/H,~\mathcal{H}^{\pi})$ and  $\mathcal{F}^{\pi}(G,\mu)$ are unitarily equivalent.
  Now, define a continuous representation $U_{\gamma}^{\pi}$ of $G$ on the continuous field of Hilbert space $\ell^{\gamma}$ corresponding to $\ell_{\gamma}^{2}(G/H,\mathcal{H}^{\pi})$
 as,
$$U_{\gamma}^{\pi}(x)= W_{r(x)}^{-1}(\operatorname{ind}_{H}^{G}(\pi),\mu)(x)W_{d(x)}. $$
 $U_{\gamma}^{\pi}(x)\in B(\ell^{\gamma}_{d(x)},\ell^{\gamma}_{r(x)})$  are unitary operators. Using the same arguments as \cite[Proposition $2.3$]{kaniuth2013induced}, we can see that $[U_{\gamma}^{\pi}(x)\eta](\omega)=\pi[\gamma(\omega)^{-1}x\gamma(x^{-1}\cdot \omega)]\eta(x^{-1}\cdot \omega), \omega \in r_{G^{0}}^{-1}(r(x))$.

Next, we show that the Hilbert bundle $\mathcal{H}^{\pi}$ is a subbundle of $\ell^{\gamma}$, for a chosen $\gamma$.
Fix $\gamma:G/H \to G$ such that, in particular, if $ q_{H}^{-1}(\omega) \subseteq H $, then
  $\gamma(\omega)=r_{G^{0}}(\omega)$ .
For $t\in C_{c}(G^{0},\mathcal{H}^{\pi})$, define
\[\eta_{t}(\omega)=\begin{cases} 
      t(r_{G^{0}}(\omega)), & q_{H}^{-1}(\omega) \subseteq H \\
      
      0~~ ,  & otherwise  
   \end{cases}
\]

With the above, we can identify $\mathcal{H}^{\pi}$ as a subbundle of $\ell^{\gamma}$ with the defined $\gamma$. Hence, due to unitary equivalence, one can easily verify that $\pi$ is a subrepresentation of $(\operatorname{ind}_{H}^{G}(\pi),\mu)_{|_{H}}$.
\end{proof}
With the above result, we show that the restriction map $u \to u_{|_{H}}$ from Fourier- Stieltjes algebra $B(G)$ of groupoid $G$ with discrete unit space $G^{0}$ to $B(H)$ of a  clopen wide subgroupoid  $H$ is surjective. 
\begin{corl}
    Let $G$ be a locally compact groupoid with discrete unit space $G^{0}$, and let $H$ be a  clopen wide subgroupoid $H$. Then given $\varphi \in P(H)$, there exist  $\phi \in P(G) $ such that $\phi_{|_{H}}=\varphi$ and $\|\phi\|=\|\varphi\|$. Moreover, the restriction map from $B(G)$ to $B(H)$ is surjective.
\end{corl}
\begin{proof}
  Define  $\phi$ such that  $\phi(x)=\varphi(x)$ for $x\in H$ and $0$ otherwise. There exist a representation $\pi$ of $H$ and $\xi \in C_{b}(G^{0},\mathcal{H}^{\pi})$ such that $\varphi(x)=\langle \pi(x)\xi(d(x)),\xi(r(x))\rangle$. By Theorem $6.1$, we can write $\phi(x)=\langle \operatorname{ind}_{H}^{G}\pi(x)\xi(d(x)),\xi(r(x))\rangle$. Using \cite[Proposition $4$]{paterson2004fourier}, we can see that $\|\phi\|=\|\varphi\|$. 
\end{proof}
We next consider a compact transitive groupoid $G$ along with a closed, transitive, wide subgroupoid $H$ with a normalized Haar system $\{\lambda_{H}^{u}\}_{u\in G^{0}}$. A result analogous to the previous one is then established with the help of the Frobenius Reciprocity Theorem \cite[Theorem $6.1$]{sridharan2025induced}. Here we will use the notion of an internally irreducible representation of the groupoid $G$. A continuous unitary representation is internally irreducible if the representation is irreducible when restricted to each isotropy subgroups. More on this can be found in  \cite[Definition $13$]{bos2011continuous}. 
A detailed discussion of the direct sum of Hilbert bundles may be found in \cite[Section 15.14]{fell1988representations}.
\begin{thm}
    Let $H$ be a closed wide transitive subgroupoid of a compact transitive groupoid $G$, and $\pi$ be a continuous unitary representation of $H$. Then $\pi$  is a subrepresentation of $(\operatorname{ind}_{H}^{G}(\pi),\mu)_{|_{H}}$, where $\mu$ is a full equivariant measure system  on $G/H$.
\end{thm}
\begin{proof}
  Each representation of a compact isotropy subgroup is direct sum of irreducible representations. So using the proof of \cite[Lemma $4.1$]{edeko2022uniform} and \cite[Theorem $4.8$]{edeko2022uniform}, one can see that $\mathcal{H}^\pi= \bigoplus_{i\in I} \mathcal{H}^{\pi_{i}}$ where $\pi_{i}$'s are internally irreducible representations of $H$. Note that each $\mathcal{H}^{\pi_{i}}$ is a finite dimensional bundle and  using \cite[Theorem $3.9(ii)$]{sridharan2025induced}, $(\operatorname{ind}_{H}^{G}(\pi),\mu)= \bigoplus_{i\in I} (\operatorname{ind}_{H}^{G}(\pi_{i}),\mu)$ . So we assume $\pi$ is an internally irreducible representation of $H$. 

 Let $\rho=(\operatorname{ind}_{H}^{G}(\pi),\mu)$, then by \cite[Theorem $4.8$]{edeko2022uniform}, there exist invariant subbundle $\rho'$ of finite dimension as subrepresentation of $\rho$. Using \cite[Theorem $6.1$]{sridharan2025induced}, there exist a $T \in Mor(\rho_{|_{H}},\pi)$. Note that $T^*$ is in $Mor(\pi,\rho_{|_H})$ due to transitivity of $H$. So $TT^{*}\in Mor(\pi,\pi)$ and by \cite[Lemma $4.9(i)$]{bos2011continuous}, choose a family of unitary operators $T \in Mor(\rho_{|_{H}},\pi)$. Hence the result follows.
 \end{proof}

 \begin{corl}
    Let $G$ be a compact transitive groupoid and $H$ a closed wide transitive subgroupoid. Then any $\varphi \in P(H)$ can be extended to $\phi \in P(G)$.
\end{corl}
\begin{proof}
    By \cite[Theorem $1$]{paterson2004fourier}, there exist a continuous unitary representation and a continuous section $\eta$ of $\mathcal{H}^{\pi_{\varphi}}$ such that $\varphi(h)= \langle\pi_{\varphi}(h)\eta,\eta\rangle$. Let $\pi=(\operatorname{ind}_{H}^{G}(\pi_{\varphi}),\mu)$. Theorem $6.3$ guarantees that $\pi_{\varphi}$ is a subrepresentation of $\pi_{|_{H}}$. So there exists a morphism $U=\{U_{u}\}_{u\in G^{0}}$ with $U_{u}$ being unitary such that $U_{r(h)}\pi_{\varphi}(h)=\pi(h)U_{d(h)}$ for all $h\in H$. 
    Hence, $\phi$ defined by $\phi(x)=\langle \pi(x)U\eta,U\eta\rangle$ extends $\varphi$.
\end{proof}
\section*{Acknowledgement}
 K. N. Sridharan is supported by the NBHM doctoral fellowship with Ref number: 0203/13(45)/2021-R\&D-II/13173.

\section*{Data Availability}
Data sharing does not apply to this article as no datasets were generated or analyzed during the current study.

\section*{competing interests}
The authors declare that they have no competing interests.
\bibliographystyle{unsrt}
\bibliography{main}
\end{document}